\documentclass[12pt]{amsart}
\usepackage{amsmath, graphicx, amsfonts,amssymb, calrsfs}
\usepackage{mathrsfs, color, amsthm}
\usepackage{mathtools}
\usepackage{enumitem}
\usepackage[utf8]{inputenc}
\usepackage{ulem}

\usepackage{tikz-cd}

\usepackage{dsfont} 

\usepackage{adjustbox}
\usepackage{tikz}

\newtheorem{theorem}{Theorem}[section]
\newtheorem{lemma}[theorem]{Lemma}
\newtheorem{proposition}{Proposition}[section]
\newtheorem{corollary}[theorem]{Corollary}

\newtheorem{claim}{Claim}[section]
\newtheorem{question}[theorem]{Question}

\newtheorem{innercustomthm}{Theorem}
\newenvironment{customthm}[1]
  {\renewcommand\theinnercustomthm{#1}\innercustomthm}
  {\endinnercustomthm}

\newtheorem{innercustomcor}{Corollary}
\newenvironment{customcor}[1]
  {\renewcommand\theinnercustomcor{#1}\innercustomcor}
  {\endinnercustomcor}  

\theoremstyle{definition}
\newtheorem{definition}[theorem]{Definition}

\numberwithin{equation}{section}

\theoremstyle{remark}
\newtheorem{remark}[theorem]{Remark}

\def\bt{\begin{theorem}}
	\def\et{\end{theorem}}
\def\bl{\begin{lemma}}
	\def\el{\end{lemma}}
\def\br{\begin{remark}}
	\def\er{\end{remark}}
\def\bc{\begin{corollary}}
	\def\ec{\end{corollary}}
\def\bd{\begin{definition}}
	\def\ed{\end{definition}}
\def\bp{\begin{proposition}}
	\def\ep{\end{proposition}}
\def\bcl{\begin{claim}}
	\def\ecl{\end{claim}}
\def\bpf{\begin{proof}}
	\def\epf{\end{proof}}
\def\be{\begin{equation}}
\def\ee{\end{equation}}
\def\bea{\begin{eqnarray}}
\def\eea{\end{eqnarray}}

\begin{document}
\title{Higher dimensional surgery and Steklov eigenvalues}
\author{Han Hong}
\address{Department of Mathematics \\
                 University of British Columbia \\
                 Vancouver, BC V6T 1Z2}
\email{honghan@math.ubc.ca}
\thanks{2010 {\it Mathematics Subject Classification.} 35P15, 58J50, 35J25. \\
This research was partially supported by the Natural Sciences and Engineering Research Council of Canada.}
\date{}
\maketitle

\begin{abstract}
We show that for compact Riemannian manifolds of dimension at least 3 with nonempty boundary, we can modify the manifold by performing surgeries of codimension 2 or higher, while keeping the Steklov spectrum nearly unchanged. 
This shows that certain changes in 
the topology of a domain do not have an effect
when considering shape optimization questions for Steklov eigenvalues in dimensions 3 and higher. Our result generalizes the 1-dimensional surgery in \cite{FS2} to higher dimensional surgeries and to higher eigenvalues. It is proved in \cite{FS2} that the unit ball does not maximize the first nonzero normalized Steklov eigenvalue among contractible domains in $\mathbb{R}^n$, for $n \geq 3$. We show that this is also true for higher  Steklov eigenvalues. Using similar ideas we show that in $\mathbb{R}^n$, for $n\geq 3$, the $j$-th normalized Steklov eigenvalue is not maximized in the limit by a sequence of contractible domains degenerating to the disjoint union of $j$ unit balls, in contrast to the case in dimension 2 \cite{GP1}.
\end{abstract}

\section{Introduction}
In this paper we study questions related to shape optimization for the Steklov eigenvalue problem on compact Riemannian manifolds of dimension at least three.
The classical theorem of R. Weinstock \cite{Wein} states that among all simply connected domains in $\mathbb{R}^2$ with fixed boundary length $2\pi$, the unit disc uniquely maximizes the first Steklov eigenvalue. It was shown in \cite{bfnt} that the Weinstock inequality holds in any dimension, provided one restricts to the class of convex sets. Namely, for every bounded convex set $\Omega \subset \mathbb{R}^n$ we have $\bar{\sigma}_1(\Omega) < \bar{\sigma}_1(\mathbb{B}^n)$, where $\bar{\sigma}_j(\Omega)=\sigma_j(\Omega)|\partial \Omega|^{1/(n-1)}$ denotes the normalized eigenvalue. However, in the wider class consisting of contractible domains, the higher dimensional analogue of Weinstock's theorem fails since there exist contractible domains $\mathbb{B}^n_{\epsilon,\delta}$ (see Figure $\ref{balltube}$) with $\bar{\sigma}_1(\mathbb{B}^n_{\epsilon,\delta})>\bar{\sigma}_{1}(\mathbb{B}^n)$ when $n \geq 3$ (\cite{FS2}). The proof involves a 1-dimensional surgery in which a small tubular neighbourhood of a curve connecting boundary components is removed from an annular domain. The construction in \cite{FS2} shows more generally that in dimensions greater than or equal to three the number of boundary components does not affect the supremum of the normalized first Steklov eigenvalue. This is in contrast to the situation in dimension two, where by adding an extra boundary component to a surface the normalized first Steklov eigenvalue can be made strictly larger (\cite{FS3}, \cite{MP}).

In this paper we consider the effect of higher dimensional surgeries on the Steklov spectrum of compact Riemannian manifolds with boundary. Specifically, we show that one can perform surgeries of codimension two or higher while keeping the normalized Steklov eigenvalues nearly unchanged. Given a compact $n$-dimensional Riemannian manifold $M$ with boundary, a compact properly embedded $m$-dimensional submanifold $\Sigma$ of $M$, and $\delta>0$ small, we let $\Omega_{\Sigma, \delta}$ denote the Lipschitz domain obtained by removing the $\delta$ tubular neighbourhood of $\Sigma$ from $M$. We call this procedure a {\it surgery of codimension $n-m$} (see Section \ref{higherdimensionalsurgery} for more precise details). Our main result is the following.

\begin{theorem} \label{higherorder}
The Steklov spectrum of a compact Riemannian manifold with boundary changes continuously under surgeries of codimension at least two, in the sense that,
if $n-m \geq 2$ then $\lim_{\delta \rightarrow 0^+} \bar{\sigma}_j(\Omega_{\Sigma,\delta}) = \bar{\sigma}_j(\Omega)$ for $j=0,1,2, \ldots$.
\end{theorem}

One purpose of such surgeries is to simplify the topology of the manifold. 
When $m=1$ and $j=1$ this theorem was proved in \cite{FS2}, in which case the surgery can be applied to construct a manifold with connected boundary from any given manifold with boundary, while keeping the first normalized Steklov eigenvalue nearly unchanged. Theorem 1.1 implies more generally that the supremum of the $j$-th normalized eigenvalue among all manifolds is the same as the supremum among manifolds having relatively simple topology; that is, among manifolds which can be obtained by performing surgeries up to codimension two. Specifically, our result implies that given any compact Riemannian manifold $\Omega$ of dimension $n \geq 3$, and given any $\epsilon >0 $ and $k \in \mathbb{N}$, there exists a smooth subdomain $\tilde{\Omega}$ of $\Omega$ such that the homomorphism of the $m$-th homology groups $i_*: H_m(\partial \tilde{\Omega})\rightarrow H_m(\tilde{\Omega})$ induced by the inclusion $i: \partial \tilde{\Omega} \hookrightarrow \tilde{\Omega}$ is injective for $m=0, 1, \ldots \min\{n-3, \left\lfloor \frac{n}{2}\right\rfloor \}$, and such that $|\bar{\sigma}_j(\tilde{\Omega})-\bar{\sigma}_j(\tilde{\Omega})|<\epsilon$ for $j=1, \ldots, k.$ This conclusion involves some standard but rather involved algebraic topology, and we omit the proof since our results here are analytic in nature.

An immediate consequence of Theorem $\ref{higherorder}$ is that for $ n\geq 3$, balls do not maximize higher normalized Steklov  eigenvalues among contractible domains in $\mathbb{R}^n$.

\begin{corollary}\label{maintheorem1}
For $n \geq 3$, there exists a contractible domain $\Omega^*$ in $\mathbb{R}^n$ such that for any $j \geq 1$,
$\bar{\sigma}_j(\Omega^*)>\bar{\sigma}_j(\mathbb{B}^n)$. 
\end{corollary}

This is not surprising, especially in light of the fact that a disc does not maximize higher Steklov eigenvalues among simply connected domains in $\mathbb{R}^2$. The classical result \cite{HPS} gives the upper bound $\bar{\sigma}_j(\Omega)\leq 2\pi j$ for any simply connected domain $\Omega$ in $\mathbb{R}^2$ and all $j \geq 0$. When $j=1$, equality is characterized in Weinstock's theorem (\cite{Wein}), but for $j \geq 2$ the inequality is strict (\cite{GP1}, \cite{FS2019P}). However, it was shown in \cite{GP1} that the inequality is sharp and is achieved in the limit by a sequence of simply connected domains degenerating to $j$ identical discs. By analogy with the $n=2$ case, it is natural to ask whether a similar result is true in higher dimensions. Using ideas from Theorem \ref{higherorder} we show that this is not the case.

\begin{theorem} \label{answer}
For $n \geq 3$ and $j \geq 2$, the supremum of the $j$-th normalized Steklov eigenvalue among contractible domains in $\mathbb{R}^n$ is not achieved in the limit by a sequence of contractible domains degenerating to the disjoint union of $j$ identical round balls.
\end{theorem}

This is a direct consequence of a more general theorem on the convergence of Steklov eigenvalues of overlapping domains as they are pulled apart, Theorem $\ref{overlappingconvergence}$, discussed in Section \ref{ballnotmaximizer}.

The paper is organized as follows. In section $\ref{higherdimensionalsurgery}$ we recall basic knowledge about the Steklov eigenvalue problem and prove our main result, Theorem $\ref{higherorder}$, on higher dimensional surgeries and the Steklov spectrum. In section $\ref{ballnotmaximizer}$ we discuss applications and related results in connection with questions about shape optimization for the Steklov problem for domains in $\mathbb{R}^n$, and prove Corollary \ref{maintheorem1}, Theorem $\ref{overlappingconvergence}$, and Theorem $\ref{answer}$.   

\vskip 2mm   \noindent {\bf Acknowledgments.} The author would like to thank his advisors Jingyi Chen and Ailana Fraser for helpful discussions.

\section{Continuity of Steklov eigenvalues under codimension 2 surgeries}  \label{higherdimensionalsurgery}

In this section, we recall some basic facts about the Steklov eigenvalue problem, and prove our main theorem on the continuity of eigenvalues under certain higher dimensional surgeries, up to codimension 2. Let $(\Omega,g)$ be a compact, connected $n$-dimensional smooth Riemannian manifold with nonempty boundary.
The Steklov eigenvalue problem is given by
\[
     \begin{cases} \Delta u=0 & \text{in}\ \Omega\\
      \frac{\partial u}{\partial \nu}=\sigma u & \text{on}\ \partial \Omega
       \end{cases}
\]
where $\nu$ is the outer unit normal to $\Omega$. When the trace operator $T: H^1(\Omega)\rightarrow L^2(\partial \Omega)$ is compact, in particular, if the domain has Lipschitz boundary, the spectrum of the Steklov eigenvalue problem is discrete 
\[
        0=\sigma_0<\sigma_1(\Omega)\leq \sigma_2(\Omega)\leq \cdots
\]
and the eigenvalues  have a standard Rayleigh quotient characterization
\begin{equation} \label{characterization}
     \sigma_j(\Omega)
     =\inf \left \{ \frac{\int_\Omega |\nabla u|^2}{\int_{\partial\Omega}u^2} \ : \
        0 \neq u\in H^1(\Omega), \ \int_{\partial \Omega}u\phi_i=0 \mbox{ for } i=0,\ldots, j-1  \right \}
 \end{equation}
where $\{\phi_0, \phi_1, \phi_2, \ldots \}$ is a complete orthonormal basis of $L^2(\partial \Omega)$ such that $\phi_i$ is an eigenfunction with eigenvalue $\sigma_i(\Omega)$ for each $i=1,\,2,\ldots$. Alternatively,
\begin{equation} \label{min-max}
       \sigma_j(\Omega) 
       =\inf_{E_{j+1}}\ \sup_{u\in E_{j+1}\setminus \{0\}} \frac{\int_\Omega|\nabla u|^2}{\int_{\partial \Omega} u^2}
\end{equation}
where the infimum is taken over all $(j+1)$-dimensional subspaces $E_{j+1}$ of the Sobolev space $H^1(\Omega)$.

We consider a smooth $m$-dimensional submanifold $\Sigma$ embedded in the manifold $\Omega$, with boundary $\partial \Sigma$ embedded in $\partial \Omega$, and meeting $\partial \Omega$ orthogonally along $\partial \Sigma$.
Let
\[
         L_{\delta}=\{x\in \Omega: d(x,\Sigma) < \delta\} \ \ \text{and}\ \ 
         T_{\delta}=\{x\in \Omega: d(x,\Sigma)=\delta\}. 
\]
Denote
\[
     \Omega_\delta=\Omega\setminus L_{\delta}. 
\]

\begin{definition}\label{definition}
Taking out the tubular neighborhood $L_{\delta}$ of $\Sigma$ gives a domain $\Omega_\delta$. Throughout the paper, we call this procedure  {\it surgery of dimension m} or {\it surgery of codimension n-m}. 
\end{definition}

We will apply the following ``no concentration" lemma to show that if $2\leq m\leq n-2$, then the interior and boundary $L^2$ norms of a sequence of eigenfunctions  don't concentrate near the neck $T_{\delta}$ as $\delta\rightarrow 0^+$. When $m=1$, that is when $\Sigma$ is a smooth curve, this phenomenon was verified in \cite[Lemma 4.2]{FS2}. We prove here that it is in fact still true provided $2\leq m\leq n-2$.

\begin{lemma} \label{vanish}
Suppose $1\leq m\leq n-2$. There exists a constant $r_0>0$ such that if $u_{\delta}\in W^{1,2}(\Omega_{\delta})$ and
\[
     \int_{\Omega_{r_0/2} \setminus \Omega_{r_0}} u_\delta^2 
     +\int_{\Omega_\delta \setminus \Omega_{r_0}}|\nabla u_\delta|^2 \leq C
\]
for $\delta\in(0,r_0/2)$, with $C$ independent of $\delta$. Then 
\begin{equation}\label{neck estimate}
      \lim_{\delta\rightarrow 0}\|u_{\delta}\|_{L^2(T_{\delta})}=0,
\end{equation}
\begin{equation}\label{interior annulus estimate}
      \lim_{s\rightarrow 0} \|u_\delta\|_{L^2(\Omega_\delta\setminus \Omega_s)}=0,
\end{equation}
\begin{equation}\label{boundary annulus estimate}
    \lim_{s \rightarrow 0}\|u_{\delta}\|_{L^2((L_s\setminus L_\delta)\cap \partial \Omega)}=0
\end{equation}
for any $0<\delta<s<r_0/2$.
\end{lemma} 
\begin{proof}
For simplicity of presentation we will assume that the normal bundle of $\Sigma$ is trivial; otherwise, we can consider local trivializations and add up all of the corresponding estimates.  Choose $r_0>0$ sufficiently small such that the exponential map of the normal bundle of $\Sigma$ is a diffeomorphism from the $r_0$ neighbourhood of the zero section onto its image, and such that the metric on the $r_0$ tubular neighbourhood $L_{r_0}$ of $\Sigma$ in $M$ is uniformly equivalent to the product metric $\tilde{g}+dr^2+r^2g_{S^{n-m-1}}$ on  $\Sigma\times D_{r_0}$, where $\tilde{g}$ is the metric on $\Sigma$ induced from $g$, $g_{S^{n-m-1}}$ is the standard metric on the sphere $S^{n-m-1}$ and $D_t$ is the ball of radius $t$ centred at the origin in $\mathbb{R}^{n-m}$. Since $\Sigma$ intersects $\partial\Omega$ orthogonally along $\partial \Sigma$, we may further choose $r_0$ sufficiently small such that the metric on $L_{r_0}\cap \partial \Omega$ is uniformly equivalent to the product metric $g_{\partial\Sigma}+dr^2+r^2g_{S^{n-m-1}}$.

We first localize the support of $u_\delta$ to lie near $\Sigma$. Choose a smooth radial cutoff function such that
\[ 
         \varphi(x,r,\theta)=\begin{cases} 1 \ \ \ r\leq \frac{r_0}{2} \\
                                                             0 \ \ \ r\geq r_0
         \end{cases}
\]
and define $v_{\delta}=\varphi u_{\delta}$. It follows from Schwarz and arithmetic geometric mean inequalities that 
\[
          |\nabla v_\delta|^2\leq 2(\varphi^2|\nabla u_\delta|^2+u_\delta^2|\nabla \varphi|^2).
\]
Since $|\nabla \varphi|$ can be bounded in terms of $r_0$ which is a fixed number, we have
\begin{equation}\label{uniform energy bound}
        \int_{\Omega_\delta}|\nabla v_\delta|^2 
        \leq  2\int_{\Omega_\delta \setminus \Omega_{r_0}}|\nabla u_\delta|^2
             +2C_1\int_{\Omega_{r_0/2}\setminus \Omega_{r_0}}u^2_{\delta} 
        \leq C_2
\end{equation}
where $C_1$, $C_2$ are constants depending only on $C$ and $r_0$. We will choose $\delta$ much smaller than $r_0$. From the construction we have that $u_\delta=v_\delta$ on $T_{\delta}$, and thus to prove equation $(\ref{neck estimate})$ in the lemma it suffices to prove that for any $\varepsilon>0$
\[
        \int_{T_{\delta}} u_\delta^2
        =\int_{T_{\delta}}v_{\delta}^2
        \leq \varepsilon\int_{\Omega_{\delta}}|\nabla v_\delta|^2
\]
for sufficiently small $\delta$. Since the metric on $\Omega_\delta$ is uniformly equivalent to the product metric on the support of $v_\delta$, it suffices to prove the estimate for the product metric.

For a fixed point $p$ in $\Sigma$ we denote the restriction of $v_{\delta}(x,r,\theta)$ to the annulus $D_{r_0} \setminus D_{\delta}$ in $\mathbb{R}^{n-m}$ at this point $p$ by $v(r,\theta)$. Choose a harmonic function $h$ on $D_{r_0} \setminus D_{\delta}$ as follows
\begin{equation*}
      \begin{cases} \Delta h=0 & D_{r_0}\setminus D_\delta \\
                                   h=v=0 & \partial D_{r_0}\\
                                   h=v  & \partial D_{\delta}.
      \end{cases}
\end{equation*}
Since harmonic functions minimize the Dirichlet energy we have
\begin{equation} \label{aa}
        \int_{D_{r_0}\setminus D_\delta}|\nabla h|^2
        \leq \int_{D_{r_0}\setminus D_\delta}|\nabla v|^2. 
\end{equation}
For any $\sigma$ with $\delta\leq \sigma\leq r_0$ we have
\begin{align*}
    \int_{D_{r_0}\setminus D_\sigma}\Delta h^2
        &=\int_{\partial D_{r_0}}\frac{\partial h^2}{\partial r}-\int_{\partial D_{\sigma}}\frac{\partial h^2}{\partial \sigma}\\
        &=-\int_{\partial D_{\sigma}}\frac{\partial h^2}{\partial \sigma}\\
        &=-\sigma^{n-m-1}\frac{d}{d\sigma}\left[\sigma^{-n+m+1}\int_{\partial D_\sigma}h^2\right] 
\end{align*}
where last equality follows since the volume measure on $\partial D_\sigma$ is $\sigma^{n-m-1}$ times that on the unit sphere $\partial D_1$. Since $\Delta h=0$, we have
\[
      \int_{D_{r_0}\setminus D_\sigma}\Delta h^2 = 2\int_{D_{r_0}\setminus D_\sigma} |\nabla h|^2,
\]
which together with $(\ref{aa})$ implies
\begin{align*}
     -\sigma^{n-m-1}\frac{d}{d\sigma} \left [ \sigma^{-n+m+1}\int_{\partial D_\sigma}h^2 \right ]
      &= 2\int_{D_{r_0}\setminus D_\sigma}|\nabla h|^2\\
      &\leq 2\int_{D_{r_0}\setminus D_\delta}|\nabla h|^2\\
      &\leq 2\int_{D_{r_0}\setminus D_\delta}|\nabla v|^2.
\end{align*}
Now dividing both sides by $\sigma^{n-m-1}$ and integrating with respect to $\sigma$ over the interval $[\delta,r_0]$ we obtain 
\[
         \delta^{-n+m+1}\int_{\partial D_\delta}v^2
         \leq 2 \left(\int_\delta^{r_0}\sigma^{-n+m+1}\ d\sigma \right) \int_{D_{r_0}\setminus D_\delta} |\nabla v|^2.
\]
If $1\leq m\leq n-2$, this implies that
\[
       \int_{\partial D_{\delta}}v^2  \leq \varepsilon_n(\delta)\int_{D_{r_0}\setminus D_\delta}|\nabla v|^2
\]
where $\varepsilon_{m+2}(\delta)=2\delta\ln(\frac{r_0}{\delta})$ and $\varepsilon_{n}(\delta)=\frac{2\delta}{n-m-2}$ for $n\geq m+3$. Integrating the above inequality over $\Sigma$ we obtain
\begin{equation} \label{equation:tube-estimate}
       \int_{T_{\delta}}v_{\delta}^2
       \leq \varepsilon_n(\delta)\int_{\Sigma^m}\int_{D_{r_0}\setminus D_\delta}|\nabla v|^2
       \leq \varepsilon_n(\delta)\int_{\Omega_\delta}|\nabla v_\delta|^2,
\end{equation}
since $|\nabla v|\leq |\nabla v_\delta|$ because $\nabla v$ is the Euclidean gradient on the slice $D_{r_0} \setminus D_\delta$ at a point on $\Sigma$. Since $\lim_{\delta\rightarrow 0}\varepsilon_n(\delta)=0$ for $1\leq m\leq n-2$, this completes the proof of $(\ref{neck estimate})$.

In what follows we shall prove $(\ref{interior annulus estimate})$ and $(\ref{boundary annulus estimate})$ of the lemma. By the same argument used to prove (\ref{equation:tube-estimate}), we have that for any $\delta\leq t<r_0/2$
\[ 
      \int_{T_t}v_{\delta}^2\leq \varepsilon_n(t)\int_{\Omega_\delta}|\nabla v_\delta|^2
      \leq C_2\varepsilon_n(t)
\]
where $\varepsilon_n(t)$ is defined as above, and $C_2$ is as in $(\ref{uniform energy bound})$ and is independent of $\delta$. Integrating with respect to $t$ over $[\delta,s]$ for $s<r_0/2$ yields 
\begin{equation}\label{interior L^2 bound} 
    \int_{\Omega_\delta\setminus \Omega_s}v_\delta^2=\int_{L_s\setminus L_\delta}v_\delta^2
    \leq \begin{cases} 
    C_2 \left(s^2\ln(\frac{r_0}{s})-\delta^2\ln(\frac{r_0}{\delta})+\frac{s^2-\delta^2}{2} \right), & n=m+2 \\
    C_2 \, (s^2-\delta^2)/(n-m-2), & n\geq m+3.
    \end{cases}
\end{equation}
Since the right hand sides tend to zero as $s \rightarrow 0$ for all $\delta < s$, this concludes the proof of $(\ref{interior annulus estimate})$.

Denote $\Sigma_t=\{x\in\Sigma: dist_\Sigma(x,\partial\Sigma)\leq t\}$ and denote $\partial\Sigma_t=\{x\in\Sigma: dist_{\Sigma}(x,\partial\Sigma)=t\}$ (note this is only part of the topological boundary of $\Sigma_t$). We choose $t$ sufficiently small  such that $H_{\partial \Sigma_\tau}\leq C$ for all $\tau \leq t$, where $H_{\partial \Sigma_\tau}$ is the mean curvature of $\partial \Sigma_\tau$ with respect to outer unit normal. By the coarea formula,

\[
     \int_0^t\int_{\partial \Sigma_\tau}\left(\int_{D_s\setminus D_\delta(x)}v_\delta^2\right) \ d\tau
     = \int_{\Sigma_t}\left(\int_{D_s\setminus D_\delta(x)}v_\delta^2\right)
     \leq \int_\Sigma\int_{D_s\setminus D_\delta(x)}v_\delta^2
     =\int_{L_s\setminus L_\delta}v_\delta^2.
\]
Denote
\[
     F_{s,\delta}(\tau):=\int_{\partial \Sigma_\tau}\left(\int_{D_s\setminus D_\delta(x)}v_\delta^2\right).
\]
Then $(\ref{interior L^2 bound})$ yields
\[
    \lim_{s\rightarrow 0}\int_0^t F_{s,\delta}(\tau)\ d\tau=0.
\]
Thus there exists $\tau_0\in (0,t)$ such that
\[
    \lim_{s\rightarrow 0}F_{s,\delta}(\tau_0)=0.
\]

We differentiate $F_{s,\delta}(\tau)$ to get
\begin{align*}
      \frac{d}{d\tau}F_{s,\delta}(\tau)
      &=\int_{\partial \Sigma_\tau}\left(\int_{D_s\setminus D_\delta(x)}2v_\delta\frac{d v_\delta}{d\tau}\right) 
           -\int_{\partial \Sigma_\tau}\left(\int_{D_s\setminus D_\delta(x)}v_\delta^2\right)H_{\partial\Sigma_\tau}.
\end{align*}
Integrating over $[0,\tau_0]$, applying the coarea formula, and using Cauchy-Schwarz inequality and arithmetic-geometric mean inequality result in that for any $0<\varepsilon<1$
\begin{align*}
    | F_{s,\delta}  (\tau_0)-  F_{s,\delta}(0)| 
    &= \left|\int_{\Sigma_{\tau_0}}\left(\int_{D_s\setminus D_\delta(x)}2v_\delta\frac{d v_\delta}{d\tau}\right)     
         -\int_0^{\tau_0}\int_{\partial \Sigma_\tau}
         \left(\int_{D_s\setminus D_\delta(x)}v_\delta^2\right)H_{\partial \Sigma_{\tau}}  \ d\tau\right| \\
    &\leq \varepsilon \int_{\Sigma_{\tau_0}}\left(\int_{D_s\setminus D_\delta(x)}|\nabla v_\delta|^2\right)
         +(1/\varepsilon+C)\int_{\Sigma_{\tau_0}}\left(\int_{D_s\setminus D_\delta(x)}v_\delta^2\right) \\
    &\leq C_2\varepsilon+(1/\varepsilon+C)\int_{L_s\setminus L_\delta}v_\delta^2
\end{align*}
Taking $\varepsilon=s$ and letting $s$ tend to zero, from $(\ref{interior L^2 bound})$ it follows that
$$F_{s,\delta}(0)\rightarrow 0 \ \text{as}\ \ s\rightarrow 0.$$
Hence we obtain the equation $(\ref{boundary annulus estimate})$.
\end{proof}
\

\begin{lemma}\label{higher} 
Consider the following logarithmic cut-off function along the submanifold $\Sigma^m$, where $1\leq m\leq n-2$, 
\[
       \varphi_{\delta}=\begin{cases}0 \ & r \leq \delta^2 \\ 
                                       \frac{2\ln \delta-\ln r}{\ln\delta} & \delta^2\leq r \leq \delta\\
                                       1& \delta\leq r.
                                 \end{cases}
\]
Then 
$\int_\Omega|\nabla \varphi_\delta|^2\rightarrow 0$ as $\delta\rightarrow 0^+.$
\end{lemma}

\begin{proof}
Integrating with respect to product metric gives
\[
      \int_{\Omega}|\nabla \varphi_\delta|^2 
      \leq \int_{\Sigma}\int_{D_\delta\setminus D_{\delta^2}}|\nabla \varphi_\delta|^2
       = \frac{c(n)|\Sigma|}{(\ln\delta)^2}\int_{\delta^2}^\delta r^{n-m-3}\ dr 
       = c(n) |\Sigma| \, \delta(n)
\]
where $\delta(n)=-1/\ln \delta$ when $n=m+2$ and $\delta(n)=(\delta^{n-m-2}-\delta^{2n-2m-4})/(n-m-2)(\ln\delta)^2$ when $n>m+2$. Since $\delta(n) \rightarrow 0$ as $\delta \rightarrow 0$, this completes the proof.
\end{proof}

We now prove our main theorem, that given a compact Riemannian manifold with boundary $(\Omega^n,g)$, surgeries of dimension $m$ (see Definition $\ref{definition}$) can be performed while keeping the Steklov eigenvalues nearly unchanged, provided $1 \leq m \leq n-2$.

\begin{customthm}{\ref{higherorder}} 
If $1 \leq m \leq n-2$, then $\lim_{\delta\rightarrow 0^+}\sigma_j(\Omega_\delta)=\sigma_j(\Omega)$ for $j=0, 1, 2, \ldots$.
\end{customthm}

\begin{proof}
Let $u^0_{\delta}, u^1_{\delta}, u^2_{\delta}, \ldots$ be $L^2(\partial \Omega_\delta)$-orthonormal Steklov eigenfunctions
such that $u^j_{\delta}$ is a Steklov eigenfunction of $\Omega_\delta$ with eigenvalue $\sigma_{j}(\Omega_\delta)$, 
\[
      \begin{cases} \Delta u^j_\delta=0  & \text{in}\ \Omega_\delta \\
      \frac{\partial u^j_\delta}{\partial \eta}=\sigma_j(\Omega_\delta)u^j_\delta  &\text{on}\ \partial \Omega_\delta.
      \end{cases}
\]

\begin{claim}\label{uniformbound}
For any $j\in \mathbb{N}$,  $\sigma_j(\Omega_\delta)$ is uniformly bounded from above for small $\delta.$ 
\end{claim}

\begin{proof}
Let $\{ f_1, \ldots, f_{j+1}\} \subset H^1(\Omega)$ be $j+1$ functions on $\Omega$ with support in $\Omega_{r_0}$ (where $r_0$ is as in Lemma \ref{vanish}), that are linearly independent on $\partial \Omega$. Let $E=\operatorname{span}\{f_1,\cdots,f_{j+1}\}$. Then if $\delta < r_0$, any function in $E$ is a valid test function for the min-max variational characterization (\ref{min-max}) of $\sigma_j(\Omega_\delta)$, and so
\[
     \sigma_j(\Omega_\delta)
     \leq \sup_{u\in E \setminus \{0 \}}\frac{\int_{\Omega_\delta}|\nabla u|^2}{\int_{\partial \Omega_\delta}u^2}
     = \sup_{u\in E \setminus \{0 \}}\frac{\int_{\Omega_{r_0}}|\nabla u|^2}{\int_{\partial \Omega_{r_0}}u^2}
            \leq \Lambda_j
\]
where $\Lambda_j$ is independent of $\delta$.
\end{proof}

Since $u_\delta^j$ is a Steklov eigenfunction of $\Omega_\delta$ with eigenvalue $\sigma_j(\Omega_\delta)$, 
\begin{equation} \label{energy-bound}
       \int_{\Omega_\delta} |\nabla u_\delta^j|^2 = \sigma_j(\Omega_\delta) \int_{\Omega_\delta} (u_\delta^j)^2 
       = \sigma_j(\Omega_\delta)  \leq \Lambda_j.
\end{equation}
By (\ref{energy-bound}) and since $\|u_\delta^j\|_{L^2(\partial \Omega_\delta)}=1$, by standard theory,
\begin{equation} \label{L^2-bound}
     \|u_\delta^j\|_{L^2(K)} \leq C(\Lambda_j, K)  
     \end{equation}
for any compact subset $K$ of $\Omega \setminus \Sigma$.

Elliptic boundary estimates ({\cite[Theorem 6.29]{tru}}) give bounds 
\[
       \|u_\delta^j \|_{C^{2,\alpha}(K)} \leq C \| u_\delta^j \|_{C^0(K)}
\]
for any compact subset $K$ of $\Omega \setminus \Sigma$ for all sufficiently small $\delta$, where $C=C(j,\alpha,\Lambda_j, K)$. 
By Sobolev embedding and interpolation inequalities (\cite[Theorem 5.2]{AF}, \cite[(7.10)]{tru}), \
\[
     \|u_\delta^j \|_{C^0(K)} \leq C \left( \varepsilon\|u_\delta^j\|_{C^{2}(K)} 
                                                   + \varepsilon^{-\mu} \|u_\delta^j \|_{L^2(K)} \right)
\]
where $\varepsilon >0$ can be taken arbitrarily small, $\mu>0$ depends on $n$, and $C$ depends on $K$. Hence $\|u_\delta^j \|_{C^{2,\alpha}(K)} \leq C$ with $C$ independent of $\delta$.
By the Arzela-Ascoli theorem and a diagonal sequence argument, there exists a sequence $\delta_i \rightarrow 0$ such that for all $j$, $u_{\delta_i}^j$ converges in $C^2(K)$ on compact subsets $K \subset \Omega \setminus \Sigma$ to a harmonic function $u^j$ on $\Omega \setminus \Sigma$, satisfying 
\[
       \frac{\partial u^j}{\partial \eta}=\sigma^j u^j 
       \quad \mbox{ on } \quad \partial \Omega \setminus \partial \Sigma
\]
with $\sigma^j =\lim_{i \rightarrow \infty} \sigma_j(\Omega_{\delta_i})$.

\begin{claim}\label{extension}
For each $j\geq 1\in \mathbb{N}$,  $u^j$ can be extended to a Steklov eigenfunction of $\Omega$ with eigenvalue $\sigma^j$.
\end{claim}
\begin{proof}
First observe that $u^j\in H^1(\Omega\setminus\Sigma)$. Fix the compact subset $\Omega_{\delta_N}$ for some large $N$. Then $\|u^j_{\delta_i}\|^2_{L^2(\Omega_{\delta_i})}=\|u^j_{\delta_i}\|^2_{L^2(\Omega_{\delta_N})}+\|u_{\delta_i}^j\|^2_{L^2{(\Omega_{\delta_i}\setminus \Omega_{\delta_N}})}$ is uniformly bounded independent of $i$. The first term is bounded by (\ref{L^2-bound})  and the second term is bounded by $(\ref{interior annulus estimate})$ of Lemma $\ref{vanish}$. This together with $(\ref{energy-bound})$ shows that $\|u_{\delta_i}^j\|_{H^1(\Omega_{\delta_i})}$ is uniformly bounded. Thus $u^j\in H^1(\Omega\setminus\Sigma)$.

For any function $\psi\in W^{1,2}\cap L^{\infty}(\Omega)$, define $\psi_\delta=\psi\varphi_\delta$ where $\varphi_\delta$ is defined in Lemma ${\ref{higher}}$. Since $u^j$ is a harmonic function on $\Omega\setminus \Sigma$ and satisfies $\frac{\partial u^j}{\partial \eta}=\sigma^j u^j$ on $\partial \Omega \setminus \partial \Sigma$, and $\psi_\delta$ vanishes near $\Sigma$, we have 
\[
       \int_{\Omega\setminus\Sigma}\nabla u^j\cdot \nabla \psi_\delta
       =\sigma^j\int_{\partial \Omega\setminus \partial \Sigma}u^j\psi_\delta,
\]
and so
\begin{equation} \label{newequation1}
     \int_{\Omega}\psi\nabla u^j\cdot\nabla \varphi_\delta+\varphi_\delta\nabla u^j\cdot \nabla \psi
     =\sigma^j\int_{\partial \Omega}u^j\psi_\delta.
\end{equation}
From H\"{o}lder's inequality it follows that
\[
     \left |\int_{\Omega}\psi \nabla u^j \cdot \nabla\varphi_\delta \right|
     \leq \int_{\Omega}|\psi\nabla u^j||\nabla \varphi_\delta|
     \leq \left(\int_\Omega|\psi\nabla u^j|^2\right)^{\frac{1}{2}}\left(\int_\Omega|\nabla \varphi_\delta|^2\right)^{\frac{1}{2}}.
\]
Therefore, by Lemma $\ref{higher}$
\[
       \int_{\Omega}\psi\nabla u^j\cdot\nabla\varphi_\delta\rightarrow 0  \ \ \text{as}\ \  \delta\rightarrow 0.
\]
Since $|\psi_\delta|\leq |\psi|\in L^\infty(\Omega)$ and $u^j\in H^1(\Omega\setminus \Sigma)$, taking limits both sides of $(\ref{newequation1})$ and applying the dominated convergence theorem, we obtain
\[
       \int_{\Omega}\nabla u^j\cdot \nabla \psi=\sigma^j \int_{\partial \Omega}u^j\psi
\]
which implies the claim.
\end{proof}

\begin{claim}\label{meanvaluezero}
The set of functions $\{u^j\}$ is orthonormal on the boundary of $\Omega$, i.e., 
\[
    \int_{\partial \Omega} u^ju^l=\delta_{jl}  \qquad  \mbox{ for } \; j,l\geq 0.
\]
\end{claim}
\begin{proof}  
Let $\epsilon >0$. Since $u^j$ are Steklov eigenfunctions by Claim $\ref{extension}$, and therefore smooth by the regularity theory, there exists sufficiently small $s>0$ such that $|\int_{L_s\cap\partial \Omega}u^ju^l|<\epsilon/4$.  According to H\"older's inequality and $(\ref{neck estimate})$ and $(\ref{boundary annulus estimate})$ of Lemma 2.2, we may furthermore assume that $s$ is chosen sufficiently small such that $|\int_{T_{\delta_i}}u^j_{\delta_i}u_{\delta_i}^l| < \epsilon/4$ and $|\int_{(L_s\setminus L_{\delta_i})\cap\partial\Omega}u^j_{\delta_i}u_{\delta_i}^l|<\epsilon/4$ hold for $i$ large enough.  The fact that $u^j_{\delta_i}$ uniformly converges to $u^j$ in any compact subset implies that  $|\int_{\partial \Omega\setminus L_s}u^j_{\delta_i}u_{\delta_i}^l-\int_{\partial \Omega\setminus L_s}u^ju^l|<\epsilon/4$ for $i$ large enough. 
Hence it follows that 
\begin{align*}
    \left|\int_{\partial\Omega}u^ju^l-\delta_{jl}\right|
    &<\left|\int_{\partial \Omega\setminus L_s}u^ju^l-\delta_{jl}\right|+\epsilon/4 \\
    &<\left|\int_{\partial \Omega\setminus L_s}u^j_{\delta_i}u_{\delta_i}^l-\delta_{jl}\right|+\epsilon/2 \\
    &=\left|\int_{\partial\Omega_{\delta_i}}u^j_{\delta_i}u_{\delta_i}^l
    -\int_{(L_s\setminus L_{\delta_i})\cap\partial\Omega}u^j_{\delta_i}u_{\delta_i}^l
    -\int_{T_{\delta_i}}u^j_{\delta_i}u_{\delta_i}^l-\delta_{jl}\right|+\epsilon/2 \\
    &<\left|\int_{(L_s\setminus L_{\delta_i})\cap\partial\Omega}u^j_{\delta_i}u_{\delta_i}^l\right|
    +\left|\int_{T_{\delta_i}}u^j_{\delta_i}u_{\delta_i}^l\right|+\epsilon/2\\
    &<\epsilon/4+\epsilon/4+\epsilon/2=\epsilon
\end{align*}
Since $\epsilon >0$ is arbitrary this yields the desired result.
\end{proof}

Theorem $\ref{higherorder}$ is proved if we can show that  $\sigma^j=\sigma_j(\Omega)$. This is what we shall do in the following, by applying induction. It is easy to see that $\lim_{\delta\rightarrow 0}\sigma_0(\Omega_\delta)=\sigma_0(\Omega)$ since they are always zero. Hereafter, we fix $j$ and assume that for each $0\leq i\leq j-1$, $u^i$ is a Steklov eigenfunction of $\Omega$ with eigenvalue $\sigma_i(\Omega)$. 
First, we show that $\sigma_j(\Omega)\leq \sigma^j.$ To see this, from Claim $\ref{meanvaluezero}$ it follows that
\[
              \int_{\partial \Omega} u^j=\int_{\partial \Omega}u^1u^j=\cdots=\int_{\partial \Omega}u^{j-1}u^j=0,
\]
which guarantees that $u^j$ is an admissible test function for $\sigma_j(\Omega)$, and thus $\sigma^j\geq \sigma_j(\Omega).$ To prove the theorem, we need to prove the reverse inequality.

Denote the $j$-th Steklov eigenfunction of $\Omega$ by $v$, and define
\[
    f=v-\sum_{i=0}^{j-1} \left(\int_{\partial \Omega_\delta}v u_{\delta}^i \right) u_\delta^i.
\]
In fact, $f$ is the projection of $v$ to the orthogonal complement of first $j$ eigenspaces of $\Omega_\delta$. Therefore,
\[
       \int_{\partial\Omega_\delta}f u_\delta^i=0, \ \text{for}\ i=0,\cdots,j-1.
\]
Thus $f$ is an admissible test function for $\sigma_j(\Omega_{\delta}).$

On the one hand, let's estimate the denominator of the Rayleigh quotient,
\begin{equation}\label{denominator}
     \int_{\partial \Omega_{\delta}}f^2
     =\int_{\partial \Omega_{\delta}}  v^2-\sum_{i=0}^{j-1} \left(\int_{\partial \Omega_{\delta}}v u_{\delta}^i\right)^2 .
\end{equation}
It's clear that
\[
     \lim_{\delta\rightarrow 0}\int_{\partial \Omega_{\delta}}  v^2=\int_{\partial \Omega}  v^2 
\]
and
\[
    \lim_{\delta\rightarrow 0}\int_{\partial \Omega_{\delta}}v u_{\delta}^i
    = \int_{\partial \Omega}v u^i=0\ \ \text{for} \ i=0,\cdots,j-1.
\]
On the other hand, we may estimate the numerator of the Rayleigh quotient as follows
\begin{equation} \label{numerator}
    \int_{\Omega_{\delta}}|\nabla f|^2
    = \int_{\Omega_{\delta}}|\nabla v|^2 
     + \sum_{i=1}^{j-1} \left(\int_{\partial \Omega_{\delta}} v u_{\delta}^i \right)^2
       \int_{\Omega_{\delta}} |\nabla u_{\delta}^i|^2\\ 
       -2\sum_{i=1}^{j-1}\left (\int_{\partial \Omega_{\delta}} v u_{\delta}^i \right) 
       \int_{\Omega_{\delta}}\langle\nabla v,\nabla u_{\delta}^i\rangle.
\end{equation}
Similarly,
\[
    \int_{\Omega_\delta}|\nabla u_\delta^i|^2
    =\sigma_1(\Omega_\delta)
    <C_1\]
    and
    \[
    \int_{\Omega_{\delta}}\langle\nabla v,\nabla u_{\delta}^i\rangle
    \leq C_2 \left (\int_{\Omega_\delta}|\nabla u_\delta^i|^2 \right)^{\frac{1}{2}}\leq C_3,
\]
where all the constant $C_1, C_2, C_3$ are independent of $\delta$ for small $\delta.$

In the end, we combine the estimates $(\ref{denominator})$, $(\ref{numerator})$ in the characterization $(\ref{characterization})$ and take the limit to get
\[
     \limsup_{\delta\rightarrow 0}\sigma_j(\Omega_\delta)
     \leq \lim_{\delta\rightarrow 0}\frac{\int_{\Omega_\delta}|\nabla f|^2}{\int_{\partial \Omega_\delta}f^2} 
     =\frac{\int_{\Omega}|\nabla v|^2}{\int_{\partial \Omega}v^2}=\sigma_j(\Omega).
\]
This completes the proof of the theorem.
\end{proof}

\section{Some results on shape optimization for Steklov eigenvalues in $\mathbb{R}^n$} \label{ballnotmaximizer}

For any bounded domain $\Omega \subset \mathbb{R}^n$ there are upper bounds on all normalized Steklov eigenvalues, $\sigma_j(\Omega)|\partial \Omega|^{\frac{1}{n-1}} \leq C(n) j^{\frac{2}{n}}$ (\cite{ceg}). An explicit upper bound for the first normalized Steklov eigenvalue was given in \cite[Proposition 2.1]{FS2}, however the upper bound is not expected to be sharp. It is an open question to determine the sharp upper bound.
\begin{question}[\cite{GP2}]
On which domain (or in the limit of which sequence of domains) is the supremum of $\bar{\sigma}_j(\Omega)$ over all bounded domains $\Omega \subset \mathbb{R}^n$ realized? 
\end{question}

A consequence of Theorem \ref{higherorder} is that the supremum of the $j$-th normalized eigenvalue among all domains is the same as the supremum among domains having relatively simple topology; that is, among domains which can be obtained by performing surgeries up to codimension two (as discussed in Section \ref{higherdimensionalsurgery}).

Another immediate consequence of our surgery result is that when $n \geq 3$ the unit ball is not the maximizer for $\bar{\sigma}_j(\Omega)$, even among contractible domains. This was proved for $j=1$ in \cite[Theorem 1.1]{FS2}.

\begin{customcor}{\ref{maintheorem1}}
For $n \geq 3$, there exists a contractible domain $\Omega^*$ in $\mathbb{R}^n$ such that for any $j\geq 1$,
$
     \bar{\sigma}_j(\Omega^*)>\bar{\sigma}_j(\mathbb{B}^n).
$
\end{customcor}

\begin{proof}
Let $\mathbb{B}_\epsilon^n$ denote the ball of radius $\epsilon$ centred at the origin in $\mathbb{R}^n$. It was shown in \cite[Proposition 3.1]{FS2} that for $\epsilon >0$ sufficiently small, $\bar{\sigma}_j(\mathbb{B}^n\setminus \mathbb{B}_\epsilon^n) >\bar{\sigma}_j(\mathbb{B}^n)$. We may then perform a 1-dimensional surgery on $\mathbb{B}^n\setminus \mathbb{B}_\epsilon^n$ to construct a contractible domain, while changing the Steklov eigenvalues by an arbitrarily small amount. For example, let $\mathbb{B}^n_{\epsilon,\delta}$ denote the domain obtained by removing a $\delta$-tube around a radial segment (see Figure $\ref{balltube}$). Then since $n \geq 3$, by Theorem $\ref{higherorder}$, for sufficiently small $\epsilon>0$ we have
\[ 
      \lim_{\delta\rightarrow 0}\bar{\sigma}_j(\mathbb{B}^n_{\epsilon,\delta})
      =\bar{\sigma}_j(\mathbb{B}^n\setminus \mathbb{B}_\epsilon^n)
      >\bar{\sigma}_j(\mathbb{B}^n). 
\]
Therefore, we can find a small $\delta>0$ such that $\bar{\sigma}_j(\mathbb{B}^n_{\epsilon,\delta})>\bar{\sigma}_j(\mathbb{B}^n).$  We let  $\Omega^{*}=\mathbb{B}^n_{\epsilon,\delta}$.
\end{proof}

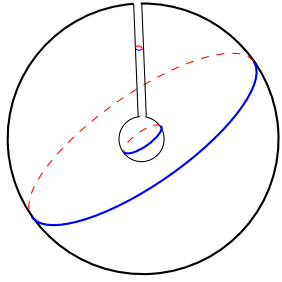
\begin{figure}[h!]
\begin{tikzpicture}[scale=0.6]

\begin{scope}

     \coordinate (0) at (0,0) ;
 
     \draw[thick] (0) + (-0.2, 3) arc (94:451:3) ;
     \draw[red,rotate=35,dashed] (0) ellipse (3 and 1) ;
 
 
     \draw[thick,blue,rotate=35] (0,0) +(180:3) arc (180:360:3 and 1) ;
     
     \draw[thin] (0)+(-0.06,0.5) arc (94:440:0.5);
     
     \draw[red,rotate=35,dashed] (0,0) ellipse (0.5 and 1/6) ;
 
 
     \draw[thick,blue,rotate=35] (0,0) +(180:0.5) arc (180:360:0.5 and 1/6) ;
     
     \draw(-0.2,3)--(-0.1,0.5);
     
     \draw(-0.02,3)--(0.08,0.5);
     
     \draw[red,thin] (0, 2) arc (20:140:0.1) ;
     
     \draw[blue,thin] (-0.15, 2) arc (225:315:0.1) ;

\end{scope}

\end{tikzpicture}
\caption{The construction of the contractible domain $\mathbb{B}^{n}_{\epsilon,\delta}$ in $\mathbb{R}^n$ for $n \geq 3$: $\delta$ is radius of the neck and $\epsilon$ is the radius of inner ball, we delete the neck and inner ball from unit ball}
\label{balltube}
\end{figure}

As discussed in the introduction, this result is not surprising, especially in light of \cite{GP1} which shows that the supremum of the $j$-th normalized eigenvalue among simply connected domains in $\mathbb{R}^2$ is achieved in the limit by a sequence of simply connected domains degenerating to the disjoint union of $j$ identical discs. 
We now consider contractible domains of this type in higher dimensions. 

Consider the necklace-like contractible domain $\Omega_{\varepsilon,l}^n$ in $\mathbb{R}^n$ which is the union of $l$ $n$-dimensional unit balls centred along a common axis and positioned such that adjacent balls overlap by a distance $\varepsilon^2$ along the axis; see Figure $\ref{figoverlapping}$. As $\varepsilon$ tends to zero, $\Omega_{\varepsilon,l}^n$ converges to $l$ identical unit balls touching tangentially at $l-1$ points, which we denote by $p_1,\cdots,p_{l-1}$  (see Figure $\ref{overlappinglimit}$). We denote $l$ disjoint unit balls by $\sqcup_l \mathbb{B}^n$.

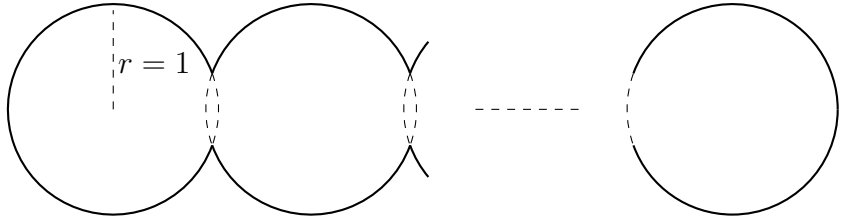
\begin{figure}[h!]
\begin{tikzpicture}[scale=0.7]

\begin{scope}
\draw [thick] (0,0) arc  (20:340:2cm);
\draw [dashed](0,0) arc  (20:-20:2cm);

\draw [thick] (0,0) arc  (160:20:2cm);
\draw [dashed](0,0) arc (160:200:2cm);
\draw [thick] (0,-1.368) arc  (200:340:2cm);
\draw [dashed](3.75877,0) arc  (20:-20:2cm);

\draw [thick] (3.75877,0) arc (160:140:2cm);
\draw [dashed] (3.75877,0) arc (160:200:2cm);
\draw [thick] (3.75877,-1.368) arc (200:220:2cm);


\draw [dashed] (5,-0.684)--(7,-0.684);

\draw [thick] (8,0) arc (160:0:2cm);
\draw [thick] (11.87938,-0.684) arc (0:-160:2cm);
\draw [dashed] (8,0) arc (160:200:2cm);

\draw[dashed] (-1.87938,-0.684)--(-1.87938,1.2);
\draw (-1.1,0.2) node{$r=1$};



\end{scope}

\end{tikzpicture}
\caption{The construction of $\Omega^n_{\varepsilon,l}$ in $\mathbb{R}^n$, $n \geq 2$: $l$ overlapping unit balls.}
\label{figoverlapping}
\end{figure}

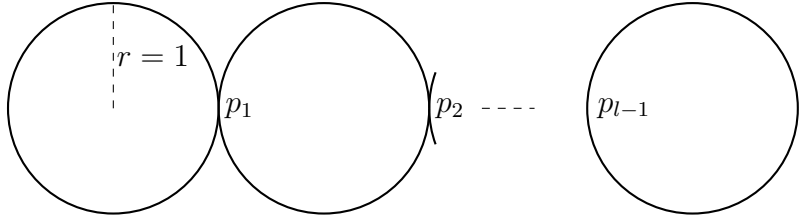
\begin{figure}[h!]
\begin{tikzpicture}[scale=0.7]

\begin{scope}
\draw [thick] (0,0) arc  (0:360:2cm);

\draw [thick] (4,0) arc (0:360:2cm);

\draw [thick] (4,0) arc (180:160:2cm);
\draw [thick] (4,0) arc (180:200:2cm);

\draw (0.4,0) node{$p_1$};
\draw (4.4,0) node{$p_2$};

\draw (7.7,0) node{$p_{l-1}$};

\draw [dashed] (5,0)--(6,0);

\draw [thick] (11,0) arc (0:360:2cm);

\draw [dashed](-2,0)--(-2,2);
\draw (-1.25,1) node{$r=1$};

\end{scope}

\end{tikzpicture}
\caption{Limit of $\Omega^n_{\varepsilon,l}$ in $\mathbb{R}^{n}$, $n \geq 2$, as $\varepsilon \rightarrow 0$.}
\label{overlappinglimit}
\end{figure}

Using similar ideas as in the proof of Theorem \ref{higherorder}, we show that the Steklov spectrum of $\Omega^n_{\varepsilon,l}$ converges to that of $\sqcup_l \mathbb{B}^n$ as $\varepsilon \rightarrow 0$.

\begin{theorem}\label{overlappingconvergence}
For every $j \geq 0$, the $j$-th normalized Steklov eigenvalue of $\Omega^n_{\varepsilon,l}$, where $n \geq 2$, converges to the $j$-th normalized Steklov eigenvalue of the disjoint union of $l$ unit balls as $\varepsilon$ tends to zero, i.e., 
\begin{equation}  \label{33333}
      \lim_{\varepsilon\rightarrow 0^+}\bar{\sigma}_j(\Omega^n_{\varepsilon,l})
      = \sigma_j(\sqcup_l \mathbb{B}^n)\cdot (l\cdot |\mathbb{S}^{n-1}|)^{\frac{1}{n-1}}.
\end{equation}
\end{theorem}

\begin{remark}\label{remark5.2}
(1) The above theorem holds for more general overlapping domains. For example, in the construction we could replace the unit ball by the domain $\mathbb{B}_{\epsilon,\delta}^n$ from the proof of Corollary \ref{maintheorem1} (Figure $\ref{balltube}$). More generally, the same proof works for domains with smooth boundary overlapping at elliptic points on the boundary, and such that other intersections do not occur; (2) Comparing $(\ref{33333})$ with Theorem 1.3.1 in \cite{GP1}, the index of the Steklov eigenvalue doesn't have to be same as the number $l$ of balls we overlap.  (3) This result was proved in \cite[Example 3]{bgt}. However, we use a different approach here.
\end{remark}

\begin{proof}[Proof of Theorem $\ref{overlappingconvergence}$]
Since $|\partial \Omega^n_{\varepsilon,l}|$ converges to $l|\mathbb{S}^{n-1}|$, it suffices to prove that
\begin{equation} \label{eigenvalue-convergence}
        \lim_{\varepsilon\rightarrow 0^+}\sigma_j(\Omega^n_{\varepsilon,l})
        =\sigma_j(\sqcup_l \mathbb{B}^n).
\end{equation}
We fix $l\geq 2$ and $n\geq 2$, and without ambiguity, we simply use the notation $\Omega_\varepsilon$ for $\Omega^n_{\varepsilon,l}$ in the proof. 
We split the proof into four steps.

\textbf{Step 1: Uniform energy bound and uniform convergence in compact subsets.} 
We consider the Steklov eigenvalue problem on $\Omega_\varepsilon$. Let $\{u_\varepsilon^j\}_{j=0}^\infty$ be $L^2(\partial \Omega_\varepsilon)$-orthonormal Steklov eigenfunctions of $\Omega_\varepsilon$ such that $u^j_{\varepsilon}$ is an eigenfunction with eigenvalue $\sigma_{j}(\Omega_\varepsilon)$. By an argument similar to the proof of Claim $\ref{uniformbound}$, for each fixed $j\in \mathbb{N}$, there exists $\varepsilon_0>0$ such that for all $\varepsilon<\varepsilon_0$,  $\sigma_j(\Omega_\varepsilon)$, and hence the energy, is uniformly bounded independent of $\varepsilon$. As in the proof of Theorem \ref{higherorder}, elliptic boundary estimates give uniform bounds on $u^j_\varepsilon$ and its derivatives up to $\partial \Omega_\varepsilon$, and there exists a sequence $\varepsilon_i \rightarrow 0$ such that for each $j$, $u^j_{\varepsilon_i}$ converges in $C^2(K)$ as $i\rightarrow \infty$ on any compact subset $K$ of $\sqcup_l \overline{\mathbb{B}^n}\setminus \{p_1, \ldots, p_{l-1}\}$ to a harmonic function $u^j$ on $ \sqcup_l \mathbb{B}^n$ satisfying
\[
    \frac{\partial u^j}{\partial \eta}=\sigma^j u^j \ \text{on}\ \sqcup_l \partial \mathbb{B}^n\setminus \{p_1, \ldots, p_{l-1}\}
\]
with $\sigma^j=\lim_{i\rightarrow \infty}\sigma_j(\Omega_{\varepsilon_i})$. 

\textbf{Step 2: Nonconcentration of $L^2$ norms of the eigenfunctions around $p_1,\cdots,p_{l-1}$.} 
For a point $p$ on the boundary of the unit ball $\mathbb{B}^n$, we consider the geodesic ball $\mathcal{B}_{s}(p)$ in $\mathbb{B}^n$ centred at $p$ with radius $s$. Notice that if $s\geq\varepsilon$ and $p$ is chosen along the axis of the domain $\Omega_\varepsilon$ then $\mathcal{B}_{s}(p)$ will contain the intersection of the adjacent overlapping balls. The boundary of $\mathcal{B}_{s}(p)$ consists of two parts: $\Gamma^1_{s}(p)=\partial\mathcal{B}_{s}(p)\cap \mathbb{B} $ and $\Gamma^2_{s}(p)=\partial\mathcal{B}_{s}(p)\setminus \Gamma^2_{s}(p)$. In the following we choose $r_0$ small such that $\mathcal{B}_{s}(p)$ is uniformly equivalent to a Euclidean half ball of radius $s$ for $s<r_0$.

We claim that for any smooth function $u$ on $\mathcal{B}_{s}(p)$ with $u=0$ on $\Gamma^1_{s}(p)$,
\begin{equation}\label{Dirichlet-Neumann}
    \int_{\mathcal{B}_{s}(p)}u^2\leq C_1(n) \, s^2 \int_{\mathcal{B}_{s}(p)}|\nabla u|^2
\end{equation}
and
\begin{equation}\label{Neumann-Steklov}
\int_{\Gamma^2_{s}(p)}u^2\leq C_2(n) \, s\int_{\mathcal{B}_{s}(p)}|\nabla u|^2.
\end{equation}
Here $C_1(n),C_2(n)$ are constants depending only on the dimension $n$. The above two Poincare inequalities follow from estimates of the first Dirichlet-Neumann eigenvalue and the first Dirichlet-Steklov eigenvalue of a Euclidean half ball (or see \cite[Lemma 4.5 \& Lemma 4.6]{FS2019P}).

According to \cite[Lemma 4.3]{FS2019P} there exists $s_1$ and a smooth cut off function $\zeta$, which is $0$ on $\mathcal{B}_{r_0}\setminus \mathcal{B}_{s_1}$ and $1$ on $\mathcal{B}_{s}$, such that for any smooth function $u$ defined on $\mathcal{B}_{r_0}$,
\begin{equation}\label{cut-off function from FS paper}
    \int_{\mathcal{B}_{r_1}\setminus \mathcal{B}_{s}}|\nabla \zeta|^2u^2\leq C(s) \left(\int_{\mathcal{B}_{r_0}\setminus \mathcal{B}_{s}}u^2+\int_{\mathcal{B}_{r_0}}|\nabla u|^2\right).
\end{equation}
Here $s<s_1=s_1(s)<r_0$ and $s_1(s)=o(1), C(s)=o(1)$. Now if $u$ is any smooth function defined on the unit ball $\mathbb{B}^n$, applying $\zeta u$ to $(\ref{Dirichlet-Neumann})$ we have
\begin{align}\label{interior small half ball estimate}
    \int_{\mathcal{B}_{s}(p)}u^2&
    \leq \int_{\mathcal{B}_{s_1}(p)}(\zeta u)^2\nonumber\\
    &\leq C_1(n)s_1^2 \int_{\mathcal{B}_{s_1}(p)}|\nabla \zeta u|^2\nonumber\\
    &\leq C_1(n)s_1^2\left(\int_{\mathcal{B}_{s_1}(p)}|\nabla u|^2
     +\int_{\mathcal{B}_{s_1}(p)\setminus \mathcal{B}_{s}(p)}|\nabla \zeta|^2u^2\right)\nonumber\\
    & \leq C_1(s) \left(\int_{\mathcal{B}_{r_0}(p)}|\nabla u|^2+\int_{\mathcal{B}_{r_0}(p)\setminus \mathcal{B}_{s}(p)}u^2\right)
\end{align}
where $C_1(s)=o(1)$, and we used $(\ref{cut-off function from FS paper})$ in the last inequality. Similarly, applying $\zeta u$ to $(\ref{Neumann-Steklov})$ we obtain
\begin{equation}\label{boundary small portion estimate}
    \int_{\Gamma^2_{s}(p)}\leq C_2(s) \left(\int_{\mathcal{B}_{r_0}(p)}|\nabla u|^2+\int_{\mathcal{B}_{r_0}(p)\setminus \mathcal{B}_{s}(p)}u^2\right)
\end{equation}
where $C_2(s)=o(1).$

Now we consider the eigenfunction $u_\varepsilon^j$ on domain $\Omega_\varepsilon.$ In what follows we shall show that interior and boundary $L^2$ norms of $u_\varepsilon^j$ don't concentrate near $p_1, \ldots, p_{l-1}$ as $\varepsilon \rightarrow 0$. The common axis of the domain $\Omega_\varepsilon$ intersects the boundaries of the overlapping balls at $2l-2$ points, $p^{(1)}_{1}(\varepsilon),p^{(2)}_1(\varepsilon),\cdots,p^{(1)}_{l-1}(\varepsilon),p^{(2)}_{l-1}(\varepsilon)$, such that as $\varepsilon\rightarrow 0$, $p_k^{(1)}(\varepsilon)$ and $p_{k}^{(2)}(\varepsilon)$ converge to $p_k$ for $k=1,\cdots,l-1$. We apply the inequality  $(\ref{interior small half ball estimate})$ at these points 
to the restriction of the function $u_\varepsilon^j$ to each ball.
Summing the inequalities yields that for any $\varepsilon<s<s_1(s)<r_0,$
\begin{align*}
\int_{\cup_{i=1}^{2}\cup_{k=1}^{l-1} \mathcal{B}_{s}(p_k^{(i)}(\varepsilon))}(u_\varepsilon^j)^2
&\leq 2C_1(s) \left(\int_{\Omega_{\varepsilon}}|\nabla u_\varepsilon^j|^2+\int_{\cup_{i=1}^2\cup_{k=1}^{l-1}\mathcal{B}_{r_0}(p_k^{(i)}(\varepsilon))\setminus \mathcal{B}_{s}(p_k^{(i)}(\varepsilon))}(u_\varepsilon^j)^2\right)\\
&\leq C \cdot \, C_1(s)
\end{align*}
where the last inequality follows from the trace inequality and the uniform energy bound. Similarly,
applying $(\ref{boundary small portion estimate})$ we obtain that for any $\varepsilon<s<s_1(s)<r_0,$
\begin{align*}
    \int_{\cup_{i=1}^2\cup_{k=1}^{l-1}\Gamma^2_{s}(p_k^{(i)}(\varepsilon))\setminus\Gamma^2_{\varepsilon}(p_k^{(i)}(\varepsilon))}(u_{\varepsilon}^j)^2&\leq
    \int_{\cup_{i=1}^2\cup_{k=1}^{l-1}\Gamma^2_{s}(p_k^{(i)}(\varepsilon))}(u_{\varepsilon}^j)^2\\
    &\leq 2C_2(s)\left(\int_{\Omega_{\varepsilon}}|\nabla u_\varepsilon^j|^2+\int_{\cup_{i=1}^2\cup_{k=1}^{l-1}\mathcal{B}_{r_0}(p_k^{(i)}(\varepsilon))\setminus \mathcal{B}_{s}(p_k^{(i)}(\varepsilon))}(u_\varepsilon^j)^2\right)\\
    &\leq C\cdot C_2(s).
\end{align*}

Hence following the ideas of the proofs of Claims \ref{extension} and $\ref{meanvaluezero}$, $\|u_{\varepsilon}^j\|_{H^1(\Omega_\varepsilon)}$ is uniformly bounded and therefore $u^j \in H^1(\sqcup_l B^n)$, $u^j$ extends to a Steklov eigenfunction of $\sqcup_l \mathbb{B}^n$, and $\{u^j\}$ are orthonormal on the boundary of $\sqcup_l \mathbb{B}^n$.

 \textbf{Step 3: Construction of cut off function.} Define the following cutoff function on $\sqcup_l \mathbb{B}^n$,
\[
     \varphi_\varepsilon=\begin{cases}
           0,  & r<\varepsilon^2 \\
           \frac{\ln r-\ln \varepsilon^2}{\ln \varepsilon-\ln \varepsilon^2} & \varepsilon^2 \leq r\leq \varepsilon  \\
           1 & r> \varepsilon
          \end{cases}
\]
where $r$ is the distance function to the nearest of the points $p_1,\ldots,p_{l-1}$.  Let $T_t=\{x\in \sqcup_l \mathbb{B}^n: r(x)\leq t\}$. We have
\begin{equation} \label{energyvanish}
   \int_{\sqcup_l \mathbb{B}^n}|\nabla  \varphi_\varepsilon|^2
   =\frac{1}{(\ln\varepsilon)^2}\int_{T_\varepsilon\setminus T_{\varepsilon^2}}\frac{1}{r^2} 
   \leq \frac{(2l-2)C(n)}{(\ln\varepsilon)^2}\int_{\varepsilon^2}^\varepsilon r^{n-3}\ dr=C(n,l)\epsilon_n(\varepsilon)
\end{equation}
where $\epsilon_{2}(\varepsilon)=-1/\ln\varepsilon$ and $\epsilon_n(\varepsilon)=\varepsilon^{n-2}(1-\varepsilon^{n-2})/(\ln\varepsilon)^2$ for $n\geq 3$.

\textbf{Step 4: Proof of the convergence of eigenvalues.} We now use induction to prove that (\ref{eigenvalue-convergence}) holds: $\lim_{\varepsilon \rightarrow 0} \sigma_j(\Omega_\varepsilon)=\sigma_j(\sqcup_l \mathbb{B}^n)$. It is clear that $\sigma_0(\Omega_\varepsilon)=0$ for any $\varepsilon$ and $\sigma_0(\sqcup_l \mathbb{B}^n )=0$, and so (\ref{eigenvalue-convergence}) holds for $j=0$. For fixed $j\in \mathbb{N}$, suppose that (\ref{eigenvalue-convergence}) holds for $i=0,1,\cdots,j-1$; that is, $u^i$ is a Steklov eigenfunction of $\sqcup_l \mathbb{B}^n$ with eigenvalue $\sigma_i(\sqcup_l \mathbb{B}^n)$ for $i=0,1,\cdots,j-1$. According to Step 1, we have
\[
      \int_{\partial (\sqcup_l \mathbb{B}^n) }u^j
      =\int_{\partial (\sqcup_l \mathbb{B}^n) }u^1u^j
      =\cdots=\int_{\partial (\sqcup_l \mathbb{B}^n) }u^{j-1}u^j=0.
\]
Then $u^j$ is an admissible test function for the $j$-th eigenvalue of $\sqcup_l \mathbb{B}^n$. Thus $\sigma_j(\sqcup_l \mathbb{B}^n)\leq \sigma^j$, and we have
\begin{equation} \label{oneinequality}
    \liminf_{\varepsilon\rightarrow 0^+}\sigma_j(\Omega_\varepsilon)
    \geq \sigma_j(\sqcup_l \mathbb{B}^n).
\end{equation} 

In order to prove the theorem, we need to prove the reverse inequality.  We will use a $j$-th eigenfunction of $\sqcup_l \mathbb{B}^n$, which we denote by $v$, to construct an admissible test function for the $j$-th eigenvalue of $\Omega_{\varepsilon^2}$, where $\varepsilon^2$ is the square of $\varepsilon$. This part of the proof is slightly different from the proof of Theorem $\ref{higherorder}$, since we need to cut off and glue functions to construct the admissible test function to approximate the $j$-th eigenvalue of $\Omega_{\varepsilon^2}$.  

Let $f=v\varphi_\varepsilon-\sum_{i=0}^{j-1}(\int_{\partial \Omega_{\varepsilon^2}} v \varphi_\varepsilon u_{\varepsilon^2}^i)u_{\varepsilon^2}^i$, where we recall that $\{u^i_{\varepsilon^2}\}_{i=0}^\infty$ are orthonormal eigenfunctions of $\Omega_{\varepsilon^2}$. By construction, $f$ is orthogonal to $u^0_{\varepsilon^2},\cdots,u^{j-1}_{\varepsilon^2}$ on $\partial \Omega_{\varepsilon^2}$, and so 
\[
      \sigma_j(\Omega_{\varepsilon^2})
      \leq \frac{\int_{\Omega_{\varepsilon^2}}|\nabla f|^2}{\int_{\partial \Omega_{\varepsilon^2}}f^2}.
\]
A simple calculation turns the numerator to
\begin{eqnarray} \label{22222222}
    \int_{\Omega_{\varepsilon^2}}|\nabla f|^2
    &=& \int_{\Omega_{\varepsilon^2}}|\nabla (v\varphi_\varepsilon)|^2
    +\sum_{i=1}^{j-1}\sigma_i (\Omega_{\varepsilon^2})
    \left( \int_{\partial \Omega_{\varepsilon^2}}v\varphi_\varepsilon u^i_{\varepsilon^2} \right)^2  \nonumber\\
    &&-2\sum_{i=1}^{j-1} \int_{\partial \Omega_{\varepsilon^2}}v\varphi_\varepsilon u^i_{\varepsilon^2}
    \int_{\Omega_{\varepsilon^2}} \langle \nabla u^i_{\varepsilon^2}, \nabla (v\varphi_\varepsilon) \rangle.
\end{eqnarray}
The first term can be estimated as 
\begin{eqnarray} \label{numeratorestimate1}
    \int_{\Omega_{\varepsilon^2}}|\nabla (v\varphi_\varepsilon)|^2
    &\leq& \int_{\Omega_\varepsilon}|\nabla v|^2
       + 2\int_{T_\varepsilon\setminus T_{\varepsilon^2}}\varphi_\varepsilon^2|\nabla v|^2
       + v^2|\nabla \varphi_\varepsilon|^2\nonumber\\
     &\leq& \int_{\sqcup_l \mathbb{B}^n}|\nabla v|^2+C_1|T_\varepsilon\setminus T_{\varepsilon^2}|
        +C_2 \int_{T_\varepsilon\setminus T_{\varepsilon^2}}|\nabla \varphi_\varepsilon|^2\nonumber\\
     &=& \int_{\sqcup_l \mathbb{B}^n}|\nabla v|^2+C_3(\varepsilon)
\end{eqnarray}
with $C_3(\varepsilon)\rightarrow 0$ as $\varepsilon\rightarrow 0^+$ by $(\ref{energyvanish})$. It's not hard to check that for $0\leq i\leq j$,
\[
     \lim_{\varepsilon\rightarrow 0^+}\int_{\partial \Omega_{\varepsilon^2}}v\varphi_\varepsilon u^i_{\varepsilon^2} 
     =\int_{\sqcup_l \mathbb{B}^n}vu^i=0 
\]
and 
\[
    \left |\int_{\Omega_{\varepsilon^2}}\langle \nabla u^i_{\varepsilon^2}, \nabla (v\varphi_\varepsilon)\rangle \right|
    \leq ( \sigma_i(\Omega_{\varepsilon^2}))^{\frac{1}{2}}
    \left( \int_{\Omega_{\varepsilon^2}}|\nabla (v\varphi_\varepsilon)|^2 \right)^{\frac{1}{2}}.
\]
Since $\sigma_i(\Omega_{\varepsilon^2})$ are uniformly bounded in $\varepsilon$, and by $(\ref{numeratorestimate1})$, the last two terms of the right hand side of $(\ref{22222222})$ both tend to zero as $\varepsilon \rightarrow 0^+$. 
In all, we have an inequality for the numerator,
\[
    \int_{\Omega_{\varepsilon^2}}|\nabla f|^2\leq \int_{\sqcup_l \mathbb{B}^n}|\nabla v|^2+C_4(\varepsilon)
\]
with $C_4(\varepsilon)\rightarrow 0$ as $\varepsilon \rightarrow 0^+$. 

On the other hand, we expand the denominator as follows
\[
     \int_{\partial \Omega_{\varepsilon^2}}f^2
     =\int_{\partial \Omega_{\varepsilon^2}} (v\varphi_\varepsilon)^2
       -\sum_{i=0}^{j-1} \left(\int_{\partial \Omega_{\varepsilon^2}}v\varphi_\varepsilon u^i_{\varepsilon^2}\right)^2
     \rightarrow \int_{\partial (\sqcup_l \mathbb{B}^n)} v^2
\]
as $\varepsilon \rightarrow 0$, since the second term tends to zero as $\varepsilon \rightarrow 0$.

Now we combine all the estimates to get
\[
    \limsup_{\varepsilon\rightarrow 0}\sigma_j(\Omega_{\varepsilon^2})
    \leq \lim_{\varepsilon\rightarrow 0} 
           \frac{\int_{\Omega_{\varepsilon^2}}|\nabla f|^2}{\int_{\partial \Omega_{\varepsilon^2}}f^2}
    \leq  \frac{\int_{\sqcup_l \mathbb{B}^n}|\nabla v|^2}{\int_{\partial (\sqcup_l \mathbb{B}^n)}v^2}
    =\sigma_j(\sqcup_l \mathbb{B}^n).
\]
Therefore together with $(\ref{oneinequality})$ we have
$\lim_{\varepsilon\rightarrow 0}\sigma_j(\Omega_\varepsilon)=\sigma_j(\sqcup_l \mathbb{B}^n)$.
\end{proof}

As a special case of Theorem \ref{overlappingconvergence}, we obtain the following generalization of \cite[Theorem 1.3.1]{GP1} to higher dimensions.

\begin{theorem} \label{theorem:j-balls}
For the domains $\Omega_{\varepsilon,j}^{n}\subset \mathbb{R}^n$ (see Figure $\ref{figoverlapping}$), $n \geq 2$, we have
\[
       \lim_{\varepsilon\rightarrow 0^+}\bar{\sigma}_j(\Omega^n_{\varepsilon,j})
       = (j\cdot |\mathbb{S}^{n-1}|)^{\frac{1}{n-1}}.
\]
\end{theorem}

\begin{proof}
When $l=j$, $\sigma_j(\sqcup_j \mathbb{B}^n)=1$, thus
$
\lim_{\varepsilon\rightarrow 0^+}\sigma_j(\Omega^n_{\varepsilon,j})
    = (j\cdot| \mathbb{S}^{n-1}|)^{\frac{1}{n-1}}. 
$
\end{proof}

Following Remark \ref{remark5.2}, if we take the domains that we overlap to be $\mathbb{B}^n_{\epsilon,\delta}$, we see that when $n \geq 3$ the domains $\Omega_{\varepsilon,j}^{n}$ do not attain the supremum of the $j$-th Steklov eigenvalue in the limit as $\varepsilon \rightarrow 0$ for any $j$, in contrast to the case in dimension two \cite{GP1}.

\begin{customthm}{\ref{answer}}
For $n \geq 3$ the supremum of the $j$-th normalized Steklov eigenvalue among contractible domains in $\mathbb{R}^n$ is not achieved in the limit by a sequence of contractible domains degenerating to the disjoint union of $j$ identical round balls.
\end{customthm}

\begin{proof}
Let $\tilde{\Omega}_{\varepsilon,j}^n$ be the domain obtained by overlapping $j$ copies of the domain $\mathbb{B}^n_{\epsilon,\delta}$. Then by the proof of Theorem $\ref{overlappingconvergence}$,
\[
      \lim_{\varepsilon\rightarrow 0^+}\bar{\sigma}_j(\tilde{\Omega}^n_{\varepsilon,j})
      =j^{\frac{1}{n-1}} \cdot \sigma_j(\mathbb{B}^n_{\epsilon,\delta})
          \cdot  |\partial \mathbb{B}^n_{\epsilon,\delta}|^{\frac{1}{n-1}}.
\]
By Theorem \ref{theorem:j-balls}, Corollary \ref{maintheorem1}, and Theorem $\ref{higherorder}$ it follows that
\[
     \lim_{\varepsilon\rightarrow 0^+}\bar{\sigma}_j(\tilde{\Omega}^n_{\varepsilon,j})
     > \lim_{\varepsilon\rightarrow 0^+}\bar{\sigma}_j(\Omega^n_{\varepsilon,j}).
\]
\end{proof}

We end this section with following corollary.
\begin{corollary}
For certain $j$ and $n$, if $\varepsilon >0$ is sufficiently small, then $\bar{\sigma}_j(\Omega_{\varepsilon, j}^n) > \bar{\sigma}_j(\mathbb{B}^n)$. 
\end{corollary}
\begin{proof}
By Theorem $\ref{theorem:j-balls}$, we have
$\liminf_{\varepsilon\rightarrow 0^+} \bar{\sigma}_j(\Omega_{\varepsilon,j}^n) = (j n\omega_n)^{\frac{1}{n-1}}>\bar{\sigma}_j(\mathbb{B}^n)$ if $j^{1/(n-1)}> \sigma_j(\mathbb{B}^n)$. By simple calculation, this only holds in certain cases. For example, when $n=3$, it holds for all $j$ except when $j$ is $4,9,16,\ldots$. For  higher $n$, it holds all $j$ less than a certain finite number depending on $n$. In these cases we can choose sufficiently small $\varepsilon$ such that $\bar{\sigma}_j(\Omega_{\varepsilon,j}^n)>\bar{\sigma}_j(\mathbb{B}^n)$. 
\end{proof}


\begin{thebibliography}{99}

\bibitem{AF} R. Adams, J. Fournier, {\it Sobolev spaces}, Second edition, Pure and Applied Mathematics
                          (Amsterdam) {\bf 140}, Elsevier/Academic Press, Amsterdam, 2003.

\bibitem{bfnt}
D. Bucur, V. Ferone, C. Nitsch and C. Trombetti, {\it Weinstock inequality in higher dimensions},
To appear in J. Differential Geom. ArXiv:1710.04587.

\bibitem{bgt}
D. Bucur, A. Giacomini and P. Trebeschi, {\it $L^\infty$ bounds of Steklov eigenfunctions and spectrum stability and domain variation}, Preprint CVGMT 2019.

\bibitem{ceg}
B. Colbois, A. El Soufi and A. Girouard, {\it Isoperimetric control of the Steklov spectrum}, J. Funct. Anal. 261 (5) (2011), 1384-1399.

\bibitem{FS3}
A. Fraser and R. Schoen, {\it Sharp eigenvalue bounds and minimal surfaces in the ball}, Inventiones Mathematicae 203 (3) (2016), 823-890.

\bibitem{FS2}
A. Fraser and R. Schoen, {\it Shape optimization for the Steklov problem in higher dimensions}, Adv. Math. 348 (2019), 146-162.

\bibitem{FS2019P}
A. Fraser and R. Schoen, {\it Some results on higher eigenvalue optimization}, 	arXiv:1910.03547.

\bibitem{tru}
D. Gilbarg and N. Trudinger, {\it Elliptic Partial Differential Equations of Second Order}, Second Edition,
Springer-Verlag (1983).

\bibitem{GP1}
A. Girouard and I. Polterovich, {\it On the Hersch-Payne-Schiffer inequalities for Steklov eigenvalues}, Funct. Anal. Appl. 44 (2010) 106-117.

\bibitem{GP2}
A. Girouard and I. Polterovich, {\it Spectral geometry of the Steklov problem}, J. Spectral Theory (2016).

\bibitem{HPS} J. Hersch, L.~E. Payne and M.~M. Schiffer, {\it Some inequalities for Stekloff eigenvalues},
               Arch. Rational Mech. Anal. 57 (1975), 99-114.
               


\bibitem{MP} H. Matthiesen and R. Petrides, {\it Free boundary minimal surfaces of any topological type in Euclidean 
              balls via shape optimization}, arXiv:2004.06051.
        
\bibitem{Wein}
R. Weinstock, {\it Inequalities for a classical eigenvalue problem}, J. Rational Mech. Anal. 3 (1954), 745-753. 
\end{thebibliography}
\end{document}